\documentclass[a4paper, 12pt, abstract]{scrartcl}
\newif\ifprivate
\privatefalse

\usepackage[utf8]{inputenc}
\usepackage[T1]{fontenc}
\usepackage{lmodern}
\usepackage[draft=false, kerning=true]{microtype}
\usepackage[english]{babel}

\usepackage{amsmath, amsthm, amsfonts}
\usepackage{mathtools}
\usepackage{physics}
\usepackage{todonotes}

\ifprivate
\usepackage[mark]{gitinfo2}
\renewcommand{\gitMark}{\jobname\,\textbullet{}\,\gitFirstTagDescribe\,\textbullet{}\,\gitAuthorName,\,\gitAuthorIsoDate}
\fi

\newcommand{\widetildealt}[1]{#1'}
\newcommand{\rowvectors}{K^{1\times D}}
\newcommand{\columnvectors}{K^{D\times 1}}
\newcommand{\rowvectorstilde}{K^{1\times \widetildealt D}}
\newcommand{\columnvectorstilde}{K^{\widetildealt D\times 1}}

\newcommand{\calA}{\mathcal{A}}
\newcommand{\N}{\mathbb{N}}
\newcommand{\C}{\mathbb{C}}

\DeclareMathOperator{\digits}{digits}
\DeclareMathOperator{\spanMO}{span}
\DeclareMathOperator{\val}{value}

\DeclarePairedDelimiter\set\{\}
\DeclarePairedDelimiterX{\setm}[2]{\{}{\}}{#1\,\delimsize\vert\,\mathopen{}#2}

\newtheorem{lemma}{Lemma}[section]
\newtheorem{proposition}[lemma]{Proposition}

\newtheorem{theorem}[lemma]{Theorem}

\theoremstyle{definition}
\newtheorem{definition}[lemma]{Definition}

\newtheorem{example}[lemma]{Example}

\theoremstyle{remark}
\newtheorem{remark}[lemma]{Remark}

\usepackage{hyperref}
\hypersetup{
  pdftex,
  a4paper,
  bookmarks,
  bookmarksopen=true,
  bookmarksnumbered=true,
  pdftitle={A Note on the Relation between Recognisable
  Series and Regular Sequences, and Their Minimal Linear
  Representations},
  pdfauthor={Clemens Heuberger, Daniel Krenn, Gabriel F. Lipnik},
  pdflang=english,
}

\begin{document}

\title{\vspace*{-2ex}
  A Note on the Relation between \\
  Recognisable Series and \\
  Regular Sequences, and Their \\
  Minimal Linear Representations%
  \footnotetext{%
  \begin{description}
  \item [Clemens Heuberger]
    \href{mailto:clemens.heuberger@aau.at}{\texttt{clemens.heuberger@aau.at}},
    \url{https://wwwu.aau.at/cheuberg},
    Alpen-Adria-Uni\-ver\-si\-tät Klagenfurt, Austria
  \item [Daniel Krenn]
    \href{mailto:math@danielkrenn.at}{\texttt{math@danielkrenn.at}},
    \url{http://www.danielkrenn.at},
    Paris Lodron University of Salzburg, Austria
  \item [Gabriel F.\ Lipnik]
    \href{mailto:math@gabriellipnik.at}{\texttt{math@gabriellipnik.at}},
    \url{https://www.gabriellipnik.at},
    Graz University of Technology, Austria
  \item [Support]
    Clemens Heuberger is supported by the Austrian Science Fund (FWF): DOC\,78.
    Gabriel F.\ Lipnik is supported by the Austrian Science Fund (FWF): W\,1230.
  \item [2020 Mathematics Subject Classification]
    11A63, 
    68Q45, 
    68R05, 
    68R15  
  \item[Key words and phrases]
    regular sequence,
    recognisable series
  \end{description}}}

\author{Clemens Heuberger, Daniel Krenn, Gabriel F.\ Lipnik}
\date{\vspace*{-3ex}}
\maketitle

\begin{abstract}
  In this note, we precisely elaborate the connection between
  recognisable series (in the sense of Berstel and Reutenauer) and
  $q$-regular sequences (in the sense of Allouche and Shallit) via
  their linear representations. In particular, we show that the
  minimisation algorithm for recognisable series can also be used to
  minimise linear representations of $q$-regular sequences.
\end{abstract}

\section{Introduction}
\subsection{Overview}\label{sec:introduction-overview}
Every regular sequence can also be seen as a recognisable series---definitions of both notions
are recalled below---and both can be described by a linear representation using a collection of square matrices and
two vectors. So when the authors of this note implemented both concepts in
SageMath~\cite{SageMath:2023:10.0}, this relation and property played fundamental roles.
For recognisable series,
there exists an algorithm to minimise the dimension of their linear
representations based on methods of
Schützenberger~\cite{Schuetzenberger:1961:recurrent-events,
  Schuetzenberger:1961:family-of-automata}; see Berstel and
Reutenauer~\cite[Chapter~2]{Berstel-Reutenauer:2011:noncommutative-rational-series}. So
it seemed to be reasonable\footnote{We did not find any reference for that. If a
  reader is aware of such a reference, please contact the authors.
  Added in proof: a remark on the topic can be found in Dumas~\cite{Dumas:2014:asymp}
  after Definition~2.} that this
algorithm can also be used for regular sequences.

When implementing the results
of~\cite{Heuberger-Krenn-Lipnik:2021:asymp-analy-recur-sequen}, the authors of this note suddenly
encountered a situation where the minimisation algorithm for recognisable
series failed for a regular
sequence.\footnote{See
  \url{https://github.com/sagemath/sage/issues/32921\#issuecomment-1418154841}. A
smaller example is presented in Example~\ref{ex:minimzation-gone-wrong}.}
At first, we were quite puzzled. It soon turned out that the linear
representation we used for our regular sequence did not fulfil a certain
eigenvector property that should be fulfilled for regular sequences, but we
were quite unsure whether fixing this completely solves the problem. The answer
is yes, and the details are the topic of this note.

\subsection{Recognisable Series and Regular Sequences}\label{section:recognisable-series-regular-sequences}
Let $\N_{0}$ denote the set of non-negative
integers and $K$ be an arbitrary field. Moreover, let $q\ge 2$ be an integer and set
$\calA_q\coloneqq \set{0, \ldots, q-1}$.

We first recall the definition of a \emph{recognisable series}; the book of Berstel and
Reutenauer~\cite[Chapter~2]{Berstel-Reutenauer:2011:noncommutative-rational-series}
provides an introduction to these series.

\begin{definition}\label{definition-recognisable-series}
  Let $\calA$ be a finite set. A sequence $x\in K^{\calA^\star}$
  is said to be a \emph{recognisable series}
  if there are a non-negative integer~$D$, a family $M=(M(a))_{a\in\calA}$
  of $D\times D$ matrices over~$K$ and
  vectors $u\in\rowvectors$, $w\in \columnvectors$ such that for all
  $b=b_0\ldots b_{\ell-1}\in\calA^\star$, we have
  \begin{equation*}
    x(b) = u M(b) w
  \end{equation*}
  with%
  \footnote{In other words, we extend the map $M\colon \calA\to K^{D\times
      D}$ to a monoid homomorphism from $\calA^\star$ to $K^{D\times D}$.
    By convention, if~$b$ is the empty word in $\calA^\star$, then
    $M(b)$ is the $D$-dimensional identity matrix.
    If the dimension~$D$ equals~$0$, then the product of
    an empty vector, $M(b)$ (for any~$b$) and another empty vector
    is an empty double sum, so equals~$0$.} 
  \begin{equation}\label{eq:morphism-extension}
    M(b)\coloneqq M(b_0) \cdots M(b_{\ell-1}).
  \end{equation}
  
  We call $(u, M, w)$ a \emph{linear representation} of~$x$ and~$D$
  the \emph{dimension} of the linear representation of~$x$.
\end{definition}

Note that we will use the convention~\eqref{eq:morphism-extension} throughout
this note.
Next, we recall the definition%
\footnote{Strictly speaking, this is an algorithmic characterisation of a
  regular sequence which is equivalent to the definition given by
  Allouche and Shallit~\cite{Allouche-Shallit:1992:regular-sequences},
  who first introduced this concept: they define a sequence~$y$ to be
  $q$-regular if the kernel
  \begin{equation*}
    \setm[\big]{y\circ (n\mapsto q^{j}n + r)}{
      \text{$j$, $r \in \N_0$ with $0\leq r < q^{j}$}}
  \end{equation*}
is contained in a finite dimensional vector space.}
of a \emph{regular sequence}; see Allouche and
Shallit~\cite{Allouche-Shallit:1992:regular-sequences,
  Allouche-Shallit:2003:autom} for characterisations, properties,
and an abundance of examples. Asymptotic properties and further examples have
been studied; cf.\ \cite{Heuberger-Krenn:2018:asy-regular-sequences},
\cite{Heuberger-Krenn-Lipnik:2021:asymp-analy-recur-sequen}, and the references therein.

\begin{definition}\label{definition:regular-sequence}
  A sequence $y\in K^{\N_0}$ is said to be \emph{$q$-regular}%
  \footnote{In the standard
  literature, the basis is frequently denoted by
  $k$ instead of our $q$ here.}
  if there are a
  non-negative integer~$D$, a family $M=(M(a))_{a\in\calA_q}$
  of $D\times D$ matrices over~$K$,
  a vector $u\in \rowvectors$ and
  a vector-valued sequence $v\in (\columnvectors)^{\N_0}$
  such that for all $n\in\N_0$, we have
  \begin{equation*}
    y(n)=uv(n),
  \end{equation*}
  and
  such that for all $r\in\calA_q$ and all $n\in\N_0$, we have
  \begin{equation}\label{equation:regular-sequence}
    v(qn+r)=M(r)v(n).
  \end{equation}

  We call $(u, M, w)$ a \emph{linear representation} of~$x$ and~$D$
  the \emph{dimension} of the linear representation of~$x$.
\end{definition}

By induction using~\eqref{equation:regular-sequence}, it is easily seen
that for all $n\in\N_0$, we have
\begin{equation}\label{eq:product-representation-of-regular-sequences}
  y(n)=uM(\digits_q(n))w
\end{equation}
where $\digits_q(n)=n_0\ldots n_{\ell-1}$ is the standard $q$ary expansion of
$n$, i.\,e., $n=\sum_{0\le j<\ell} n_jq^j$ with $n_{\ell-1}\neq 0$, and
$w\coloneqq v(0)$. In other
words, given a $q$-regular sequence $y$ with linear representation $(u, M, w)$
and considering the recognisable series $x$ with
linear representation $(u, M, w)$ over the alphabet $\calA_q$, we can write
$y=x \circ \digits_q$.

As mentioned in Section~\ref{sec:introduction-overview},
this is how the authors of this note implemented regular sequences in SageMath: these were a special
case of recognisable series---technically speaking, the class
\texttt{RegularSequence} is a subclass of the class
\texttt{RecognizableSeries}---where
accessing the values for integer values is translated accordingly and
additional properties (such as subsequences) are implemented. To construct a
regular sequence in SageMath, the input is a family of square matrices $M$ and
two vectors $u$ and $w$.

\subsection{Minimisation Failing?}

Definitions~\ref{definition-recognisable-series}
and~\ref{definition:regular-sequence} are worded in such a way as to allow several
different linear representations of the same recognisable series or regular
sequence. When working with such objects algorithmically, inevitably, the
question of minimality of the dimension of the linear representation arises.

As mentioned in Section~\ref{sec:introduction-overview}, Berstel and
Reutenauer~\cite{Berstel-Reutenauer:2011:noncommutative-rational-series}
describe an algorithm to determine a linear representation of minimal dimension
for a recognisable series, given by some linear representation. The remarks at
the end of Section~\ref{section:recognisable-series-regular-sequences} imply
that this algorithm is also used in SageMath to find a linear representation of
minimal dimension of a regular sequence. As also mentioned in
Section~\ref{sec:introduction-overview}, this led to a problem, and here is a
simplified example illustrating it.

\begin{example}\label{ex:minimzation-gone-wrong}
  Let $q=2$,
  \begin{equation}\label{eq:example-minimisation-gone-wrong-linear-representation}
    u = \begin{pmatrix}1 & 0\end{pmatrix},\quad
    M(0) = \begin{pmatrix}
      1 & 1 \\
      0 & 0
    \end{pmatrix},\quad
    M(1) = \begin{pmatrix}
      1 & 0 \\
      0 & 0
    \end{pmatrix},\qq{and}
    w = \begin{pmatrix}0 \\ 1\end{pmatrix},
  \end{equation}
  and consider the sequence $y$ defined
  by~\eqref{eq:product-representation-of-regular-sequences}
  for these values of $u$, $M$, and $w$.

  We have $y(0)=uw=0$ and for all positive integers $n$, the standard binary
  expansion $\digits_2(n)$ ends on a $1$, so writing $\digits_2(n)=b1$ for some
  $b\in\set{0,1}^{\star}$, we have
  \begin{equation*}
    y(n) = uM(b)M(1)w =
    \begin{pmatrix}1 & 0\end{pmatrix}
    M(b)
   \begin{pmatrix}
      1 & 0 \\
      0 & 0
    \end{pmatrix}
    \begin{pmatrix}0 \\ 1\end{pmatrix}
    =
    \begin{pmatrix}1 & 0\end{pmatrix}
    M(b)
    \begin{pmatrix}0\\0\end{pmatrix}
    =0.
  \end{equation*}
  Hence, we have shown $y(n)=0$ for all $n\in\N_0$. For the zero sequence, the
  minimal linear representation is the representation of dimension $0$, i.\,e.,
  the left and right vectors as well as the matrices $M(a)$ are empty and all
  matrix products, vector-matrix products, and matrix-vector products are empty
  sums and therefore $0$, as required.

  We now compare our result here with the one in SageMath. The input
  \begin{verbatim}
S = RecognizableSeriesSpace(QQ, [0, 1])

u = vector([1, 0])
M_0 = matrix([[1, 1], [0, 0]])
M_1 = matrix([[1, 0], [0, 0]])
w = vector([0, 1])

x = S([M_0, M_1], u, w)
x.minimized().linear_representation()
\end{verbatim}
  yields the output
  \begin{verbatim}
((1, 0), Finite family {0: [0 1] [0 1], 1: [1 0] [1 0]}, (0, 1)) .
\end{verbatim}
  We see that SageMath returns a linear representation of dimension 2 (not
  identical to the input linear representation) and claims it to be minimal.
  Note that we used a recognisable series here because the algorithm by Berstel
  and Reutenauer is formulated for recognisable series.
\end{example}

So we do have a problem: we easily saw that our regular sequence $y$ has a linear
representation of dimension $0$, but SageMath answered that all linear
representations of the underlying recognisable series $x$
have dimension at least $2$.

A few possibilities come to mind. Almost unthinkably, there could be an error
in the algorithm of Berstel and Reutenauer, or, more probably, in our
implementation of that algorithm in SageMath. Despite a clear peer-reviewing
policy for contributions into SageMath, we might have overlooked something.
We ask SageMath once more to get a partial answer.

\begin{example}[Continuation of Example~\ref{ex:minimzation-gone-wrong}]
  \label{ex:minimzation-gone-wrong-cont}
  The input
\begin{verbatim}
x
\end{verbatim}
  gives the first few terms of the recognisable series as
\begin{verbatim}
[0] + [00] + [10] + [000] + [010] + [100] + [110] 
    + [0000] + [0010] + [0100] + ...
\end{verbatim}
  So it seems that the recognisable series does not vanish. The output suggests
  that the recognisable series $x$ with the linear representation defined
  in~\eqref{eq:example-minimisation-gone-wrong-linear-representation} is one
  exactly for those input words with trailing zeros. Indeed,
  \begin{equation*}
    x(b0) = uM(b)M(0)w =
    \begin{pmatrix}1 & 0\end{pmatrix}
    M(b)
    \begin{pmatrix}
      1 & 1 \\
      0 & 0
    \end{pmatrix}
    \begin{pmatrix}0 \\ 1\end{pmatrix}
    =
    \begin{pmatrix}1 & 0\end{pmatrix}
    M(b)
    \begin{pmatrix}1\\0\end{pmatrix}
    = 1
  \end{equation*}
  by induction.

  This means that~$x(b)=1$ if $0$ is a suffix of $b$ and $x(b)=0$ otherwise.
  It is therefore clear that there cannot be a linear representation of
  $x$ of dimension $0$ (because that would lead to all zeros due to empty
  sums). It is not hard to see that $x$ cannot have a linear representation of
  dimension $1$, either: then the matrices $M(0)$ and $M(1)$ of that linear
  representation would forcibly commute and $x(b)$ could only depend on the
  number of occurrences of the letters $0$ and $1$ in $b$, but not on their
  position. Thus, independently of the algorithm of Berstel and Reutenauer (and
  its implementation in SageMath), we conclude that any linear representation
  of $x$ must have dimension at least $2$.
\end{example}

So with Examples~\ref{ex:minimzation-gone-wrong}
and~\ref{ex:minimzation-gone-wrong-cont}, we did not construct a
counter example to the validity of the minimisation algorithm of
Berstel and Reutenauer for recognisable series, however in this
particular example, we see that we cannot apply the algorithm to find
a minimal representation of the regular sequence. More generally, this means
that choosing an arbitrary family of matrices $M$ and left and right
vectors $u$ and $w$, respectively, and defining a regular sequence
by~\eqref{eq:product-representation-of-regular-sequences} can lead to
situations where the algorithm of Berstel and Reutenauer for recognisable
series does not return a minimal linear representation for the regular
sequence.

Is that the final word? Or can we somehow find at least some situations where
using the minimisation algorithm for recognisable series is valid for the
``corresponding'' regular sequence?

\subsection{Why this Note is Needed}\label{sec:why}

The short answer is, we need this note to discuss the questions raised
by the example above. So, let us briefly come back to Example~\ref{ex:minimzation-gone-wrong}.
A key feature was that at first we
only inserted words with trailing one (or the empty word) into the recognisable
series because standard binary expansions of positive integers have exactly
this property. And indeed, as the discussion in the examples shows, inserting
any other binary expansion (with trailing zeros) instead of the standard binary
expansion would lead to another result. This seems to be an important
distinction between recognisable series (all words allowed) and regular
sequences (only words without trailing zeros inserted into the corresponding
recognisable series).

At first glance, the following observation seems to be a technical
detail: If we insert $n=r=0$ into~\eqref{equation:regular-sequence}, we obtain
$v(0)=M(0)v(0)$. In other words, if not zero, then $v(0)$ is an eigenvector of
$M(0)$ associated with the eigenvalue $1$. It seems to be a minor detail
because once we replace the formulation~\eqref{equation:regular-sequence} by
\eqref{eq:product-representation-of-regular-sequences}, it does not seem to be
relevant any more. However, we note that this condition is not fulfilled in
Example~\ref{ex:minimzation-gone-wrong}, as
$M(0)w=\big(\begin{smallmatrix}1\\0\end{smallmatrix}\big)\neq w$. So $(u, M,
w)$ as given
by~\eqref{eq:example-minimisation-gone-wrong-linear-representation} is
\emph{not} a linear representation of the $2$-regular sequence $y$ considered
in Example~\ref{ex:minimzation-gone-wrong} (and we carefully never claimed it
to be one).\footnote{In order to construct a linear representation for a regular sequence out of a linear representation of a recognisable series as given by~\eqref{eq:example-minimisation-gone-wrong-linear-representation}, one can follow the proof of~\cite[Lemma~4.1]{Allouche-Shallit:1992:regular-sequences}.}

This raises two questions. Suppose we have a regular sequence $y$, take a linear
representation $(u, M, w)$ of that regular sequence (thus implying $M(0)w=w$),
take it as a linear representation of a recognisable series $x$, and then run the
minimisation algorithm by Berstel and Reutenauer on it. Will this approach yield a minimal
linear representation of $y$? And will that linear representation still fulfil
the essential eigenvector property?

The answer to both questions is yes. But according to the saying ``fool me
once, shame on you; fool me twice, shame on me'', we should make sure to
have a proof. The nature of SageMath as an open source software system also
means that this proof should not be a ``well-known fact in the community'' or
some kind of an urban myth, but something which can be clearly referenced.
In particular, such a reference is put in
the documentation of SageMath. This note sets out to provide that proof and to
clarify the relation between recognisable series and regular sequences,
their linear representations, and their minimal linear representations.

\subsection{Structure of this Note}
In
Section~\ref{sec:recognisable-series}, we collect information on
recognisable series and their minimal linear representations. Finally, in
Section~\ref{sec:regular-sequences}, we consider regular sequences, their
connection to recognisable series, and prove the main result of this 
paper (Theorem~\ref{theorem}) as outlined in the last two paragraphs of Section~\ref{sec:why}.
We close by Example~\ref{ex:ignore-trailing-zeros} providing another angle
on what can go wrong with minimisation.

\section{Recognisable Series}\label{sec:recognisable-series}

\begin{definition}
  A linear representation of a recognisable series~$x$ is said to be a
  \emph{minimal}%
  \footnote{In~\cite{Berstel-Reutenauer:2011:noncommutative-rational-series},
    these linear representations are called \emph{reduced} instead of
    minimal.}
  linear representation of~$x$ if its dimension is minimal over all linear
  representations of~$x$.
\end{definition}

Berstel and Reutenauer present a characterisation for minimal linear
representations~\cite[Proposition~2.1]{Berstel-Reutenauer:2011:noncommutative-rational-series}.
In this note we only need the direction of the following lemma and for reasons of
self-containedness, we give an ad-hoc proof here.

\begin{lemma}[{Berstel--Reutenauer~\cite[Proposition~2.1]{Berstel-Reutenauer:2011:noncommutative-rational-series}}]
  \label{lemma:necessary-condition-minimality}
  Let $\calA$ be a finite set, $x\in K^{\calA^\star}$ be a recognisable series
  and $(u, M, w)$ be a minimal linear representation of~$x$ of dimension~$D$.
  Then $\spanMO(\setm{u M(b)}{b\in\calA^{\star}})=\rowvectors$.
\end{lemma}

The basic idea of the proof is that if that span had lower dimension, then
everything would take place in a proper subspace. Taking matrix representations with
respect to a basis of the subspace would then give a lower dimensional linear representation.

We mention that by symmetry, we also have $\spanMO(\setm{M(b)w}{b\in
  \calA^\star})=\columnvectors$; but we will not need this property here.

\begin{proof}[{Proof of Lemma~\ref{lemma:necessary-condition-minimality}}]
  Let $S\coloneqq \setm{uM(b)}{b\in\calA^{\star}}$.
  Toward a contradiction, assume that we have $\spanMO(S)=W$ for some proper subspace~$W$ of~$\rowvectors$ of
  dimension $\widetildealt D$ and let
  $B$ be a basis of~$W$. Let $\Phi_B\colon \rowvectorstilde\to W$ be
  the coordinate map with respect to the basis~$B$.

  As an implication of the definition of~$W$, we have $u\in W$, and so
  there exists a $\widetildealt u\in \rowvectorstilde$ such
  that $\Phi_B(\widetildealt u)=u$.
  For all $a\in\calA$, the map $v\mapsto vM(a)$ is an endomorphism of~$W$ by
  construction of~$W$: for $b\in\calA^\star$, we have
  $u M(b)M(a) = u M(ba) \in S$, which implies that the 
  map under consideration maps~$S$ into itself and 
  therefore maps~$W=\spanMO(S)$ into itself.
  We now construct a family $\widetildealt M=(\widetildealt M(a))_{a\in\calA}$
  of $\widetildealt D\times \widetildealt D$ matrices over~$K$ as follows.
  For all $a\in\calA$, let $\widetildealt{M}(a)$ be the
  matrix representation of the endomorphism $v\mapsto vM(a)$ of~$W$ with respect to the
  basis~$B$, i.\,e., $\Phi_B(\widetildealt v) M(a)=\Phi_B(\widetildealt v \widetildealt
  M(a))$ holds for all $\widetildealt v\in \rowvectorstilde$. Let $\widetildealt w\in \columnvectorstilde$ be the
  matrix representation of the homomorphism $v\mapsto vw$ from~$W$ to~$K$ with
  respect to the basis~$B$ of~$W$ and the standard basis of~$K$, i.\,e.,
  $\Phi_B(\widetildealt v) w=\widetildealt v \widetildealt w$ holds for all $\widetildealt
  v\in \rowvectorstilde$.

  For all $b\in\calA^\star$, this implies that
  \begin{equation*}
    x(b)
    = u M(b) w
    = \Phi_B(\widetildealt u)M(b) w
    = \Phi_B(\widetildealt u \widetildealt M(b))w
    = \widetildealt u \widetildealt M(b)\widetildealt w.
  \end{equation*}
  In other words, $(\widetildealt u, \widetildealt M, \widetildealt w)$
  is a linear representation of~$x$, and its dimension is $\widetildealt D<D$, a
  contradiction to $(u, M, w)$ being a minimal linear representation of~$x$.
\end{proof}

Let us consider a linear representation~$(u, M, w)$ of a recognisable
series~$x \in K^{\calA^\star}$, $\calA$ a finite set.
If there is a $z\in\calA$ with $M(z)w=w$, then it is clear
that $x(bz)=x(b)$ holds for all $b\in\calA^\star$.
It turns out that the converse is true if the
linear representation is minimal. This is the assertion of the following proposition.

\begin{proposition}\label{proposition:eigenvector}Let $\calA$ be a finite set, $x\in K^{\calA^\star}$ be a
  recognisable series and $(u, M, w)$ be a minimal linear representation of~$x$.
  Let $z\in\calA$ be such that $x(bz)=x(b)$ holds for all
  $b\in\calA^\star$. Then we have $M(z)w=w$.
\end{proposition}
\begin{proof}
  Let $D$ be the dimension of the linear representation $(u, M, w)$, and
  set $S\coloneqq \setm{uM(b)}{b\in\calA^{\star}}$. As
  \begin{equation*}
    uM(b)w=x(b)=x(bz)=uM(bz)w=uM(b)M(z)w
  \end{equation*}
  holds for all $b\in \calA^\star$,
  the linear maps $v \mapsto vw$ and $v \mapsto vM(z)w$ from $\rowvectors$ to~$K$
  coincide on~$S$.
  As~$S$ generates~$\rowvectors$ by
  Lemma~\ref{lemma:necessary-condition-minimality}, these maps also
  coincide on $\rowvectors=\spanMO(S)$.
  Therefore, their matrix representations~$w$ and~$M(z)w$ coincide.
\end{proof}

\begin{definition}
  Let $x\in K^{\calA_q^\star}$ be a recognisable series such that
  $x(b0)=x(b)$ holds for all $b\in\calA_q^\star$. Then~$x$ is said to be
  \emph{compatible with regular sequences} (or simply \emph{compatible}).
\end{definition}
\begin{remark}\label{remark:eigenvector}
  Let~$x$ be a recognisable series with minimal
  linear representation~$(u, M, w)$. Then by Proposition~\ref{proposition:eigenvector},
  $x$ being compatible is equivalent to the condition~$M(0)w=w$.
\end{remark}

The following example shows, however, that non-minimal linear
representations~$(u, M, w)$ of a compatible recognisable series do not
necessarily satisfy the property~$M(0)w=w$.

\begin{example}
  Consider the constant recognisable series
  $x\in\C^{\set{0,1}^{\star}}$ with $x(b) = 1$ for all
  $b\in\set{0,1}^{\star}$. It is clear that~$x$ is compatible and a
  minimal linear representation $(u, M, w)$ is given by
  \begin{equation*}
    u = (1)\in \mathbb{C}^{1\times 1},\quad M(0) = M(1) = (1) \in\mathbb{C}^{1\times 1}\qq{and} w = (1) \in \mathbb{C}^{1\times 1}.
  \end{equation*}
  So $M(0)w = w$ holds, as stated in
  Remark~\ref{remark:eigenvector}.

  Moreover, $(u', M', w')$ with
  \begin{equation*}
    u' = \begin{pmatrix} 1 & 0\end{pmatrix}, \quad
    M'(0) = M'(1) = \begin{pmatrix} 1 & 0\\ 0 & 2 \end{pmatrix}
    \qq{and}
    w' = \begin{pmatrix} 1\\ 1 \end{pmatrix}
  \end{equation*}
  is also a linear representation of~$x$: for a word
  $b\in\set{0,1}^{\star}$ of length~$\ell$, we have
  \begin{align*}
    u'M'(b)w' &=
    \begin{pmatrix} 1 & 0\end{pmatrix}
    \begin{pmatrix} 1 & 0\\ 0 & 2 \end{pmatrix}^{\ell}
    \begin{pmatrix} 1\\ 1 \end{pmatrix}\\ &=
    \begin{pmatrix} 1 & 0\end{pmatrix}
    \begin{pmatrix} 1 & 0\\ 0 & 2^{\ell} \end{pmatrix}
    \begin{pmatrix} 1\\ 1 \end{pmatrix}\\ &=
    \begin{pmatrix} 1 & 0\end{pmatrix}
    \begin{pmatrix} 1\\ 2^{\ell} \end{pmatrix} = 1 = x(b);
  \end{align*}
  the lower right entry in $M'(b)$ is annihilated by the
  zero in~$u'$. However, $M'(0)w' = w'$ does not hold. This is no
  contradiction to Remark~\ref{remark:eigenvector} because
  $(u', M', w')$ is not minimal.
\end{example}

\section{Regular Sequences}\label{sec:regular-sequences}
\begin{definition}
  A linear representation of a regular sequence~$y$ is said to be
  a \emph{minimal}
  linear representation of~$y$ if its dimension is
  minimal over all linear representations of~$y$.
\end{definition}

The first statement of the following lemma corresponds
to~\cite[Lemma~4.1]{Allouche-Shallit:1992:regular-sequences} and has already
been discussed in \eqref{eq:product-representation-of-regular-sequences}; the
second statement has been discussed towards the end of
Section~\ref{sec:why}. Nevertheless, we restate it here for completeness
and refer to the mentioned references for proofs.

\begin{lemma}\label{lemma:regular-sequence-as-recognisable-series}
  Let $y\in K^{\N_0}$ be a $q$-regular sequence with linear representation $(u,
  M, w)$ and let $n\in \N_0$. Then we have
  \begin{equation}\label{equation:regular-sequence-as-recognisable-series}
    y(n)=u M(\digits_q(n)) w.
  \end{equation}
  Furthermore, we have $M(0)w=w$.
\end{lemma}

In the following lemma, for $b=a_0\ldots a_{\ell-1}\in\calA_q^\star$, we set
\begin{equation*}
  \val(b)\coloneqq \sum_{j=0}^{\ell-1} a_j q^j,
\end{equation*}
with the usual convention that if~$b$ is the empty word in $\calA_{q}^\star$, then
$\val(b)=0$.

\begin{lemma}\label{lemma:associated-recognisable-series-well-defined}
  Let $y\in K^{\N_0}$ be a $q$-regular sequence and $(u, M, w)$ a
  linear representation of~$y$. Then
  \begin{equation*}
    y(\val(b))=uM(b)w
  \end{equation*}
  holds for all $b\in\calA_q^\star$. In particular, the value of
  $uM(b)w$ is independent of the particular choice of the linear
  representation~$(u, M, w)$ of~$y$ and of trailing zeros of~$b$.
\end{lemma}
\begin{proof}
  Let $b\in\calA_q^\star$ and write $b=c0^\ell$ for some $c\in\calA_q^\star$
  and some $\ell\in\N_0$ such that~$\ell$ is
  maximal. Then $\val(b)=\val(c)$, and
  by~\eqref{equation:regular-sequence-as-recognisable-series}, we have
  \begin{equation*}
    y(\val(c)) = uM(c)w.
  \end{equation*}
  As $M(0)w=w$ by
  Lemma~\ref{lemma:regular-sequence-as-recognisable-series}, we also have
  \begin{equation*}
    uM(b)w = uM(c)M(0)^\ell w=uM(c)w,
  \end{equation*}
  as required.
\end{proof}

\begin{definition}
  Let $y\in K^{\N_0}$ be a $q$-regular sequence with linear representation $(u,
  M, w)$. Then the recognisable series $x\in K^{\calA^\star}$ with linear
  representation $(u, M, w)$ is called the recognisable series \emph{associated} to~$y$.
\end{definition}

\begin{remark}\label{remark:associated-recognisable-series-compatible}
  From Lemma~\ref{lemma:associated-recognisable-series-well-defined} we see
  that the recognisable series associated to a $q$-regular sequence is well defined.
  From Lemma~\ref{lemma:regular-sequence-as-recognisable-series} we see that a
  recognisable series associated to a $q$-regular sequence is compatible.
  Moreover, we see that every linear representation of a $q$-regular
  sequence is also a linear representation of its associated
  recognisable series.
\end{remark}

\begin{theorem}\label{theorem}
  Let $y$ be a $q$-regular sequence and $(u, M, w)$ be a minimal linear
  representation of the recognisable series associated to~$y$. Then $(u, M, w)$
  is a linear representation of~$y$, and it is also minimal.
\end{theorem}

In other words, to find a minimal linear representation of a regular
sequence, we can use the minimisation algorithm presented by
Berstel and Reutenauer~\cite[Chapter~2]{Berstel-Reutenauer:2011:noncommutative-rational-series}
on the associated recognisable series, i.\,e., the
recognisable series with the same linear representation as the regular
sequence.

\begin{proof}[Proof of Theorem~\ref{theorem}]
  Let~$x$ denote the recognisable series associated to~$y$
  and~$D$ denote the dimension of $(u, M, w)$.
  By Remark~\ref{remark:associated-recognisable-series-compatible}, $x$ is
  compatible, and therefore, by Proposition~\ref{proposition:eigenvector}
  (see Remark~\ref{remark:eigenvector}), we have
  $M(0)w=w$.

  We define a vector-valued sequence $v\in (\columnvectors)^{\N_0}$ by
  $v(0)\coloneqq w$ and~\eqref{equation:regular-sequence} for all $n\in\N_0$
  and $r\in\calA_q$.
  Note that $M(0)w=w$ implies the validity of~\eqref{equation:regular-sequence}
  for $n=0$ and $r=0$.
  By the above definition of~$v$ and
  Lemma~\ref{lemma:associated-recognisable-series-well-defined},
  $(u, M, w)$
  is indeed a linear representation of~$y$.

  Now, any linear representation of the $q$-regular sequence~$y$
  of dimension~$\widetildealt D$
  is also a linear representation of the recognisable series~$x$ by Remark~\ref{remark:associated-recognisable-series-compatible}.
  Therefore, due to minimality of $(u, M, w)$, we have $\widetildealt D\ge D$.
  In particular, by choosing a minimal linear representation of~$y$, we see
  that $(u, M, w)$ is a minimal linear representation of~$y$ as well.
\end{proof}

At last, we can relax the assumptions of Theorem~\ref{theorem} and
ask: Given a $q$-regular sequence with minimal linear representation,
can we find a recognisable series that gives the same values for each
standard $q$ary expansion of a non-negative integer, but whose minimal
linear representation has a smaller dimension than that of the regular
sequence? The following example provides an affirmative answer.

\begin{example}\label{ex:ignore-trailing-zeros}
  Let us consider the recognisable series
  $x\in\C^{\set{0,1}^{\star}}$ with $x(b) = 2^t$ and $t$ counting
  the letter~$0$ in $b\in\set{0,1}^{\star}$.
  Note that $x(b0) = 2 x(b) \neq x(b)$ for all $b\in\set{0,1}^{\star}$;
  in particular, $x$ is not compatible.
  Moreover, a minimal linear representation $(u, M, w)$ of~$x$ is given by
  \begin{equation*}
    u = (1)\in \mathbb{C}^{1\times 1}, \quad
    M(0) = (2) \in\mathbb{C}^{1\times 1}, \quad
    M(1) = (1) \in\mathbb{C}^{1\times 1}
    \qq{and} w = (1) \in \mathbb{C}^{1\times 1}.
  \end{equation*}

  In contrast, let $y\in\C^{\N_0}$ be the $2$-regular sequence with
  $y(n) = x(\digits_2(n))$ for all $n\in\N_0$.
  Then $(u', M', w')$ with
  \begin{equation*}
    u' = \begin{pmatrix} 1 & 0\end{pmatrix}, \quad
    M'(0) = \begin{pmatrix} 2 & -1\\ 0 & 1 \end{pmatrix}, \quad
    M'(1) = \begin{pmatrix} 1 & 0\\ 0 & 0 \end{pmatrix}
    \qq{and}
    w' = \begin{pmatrix} 1\\ 1 \end{pmatrix}
  \end{equation*}
  is a minimal linear representation of~$y$.

  Therefore, starting with the $2$-regular sequence~$y$ whose
  minimal representation has dimension~$2$ can lead to
  a minimal representation of dimension~$1$ of a recognisable series
  when ignoring trailing zeros.
\end{example}


\bibliography{cheub-bib/cheub}

\providecommand{\Submitted}{Submitted} \providecommand{\availableat}{ available
  at } \providecommand{\alsoavailableat}{ also available at }
  \providecommand{\evavailableat}{earlier version available at }
  \providecommand{\toappearin}{To appear in } \providecommand{\toappear}{to
  appear} \providecommand{\inpreparation}{in preparation}
  \providecommand{\doi}[1]{\href{http://dx.doi.org/#1}{\path{doi:#1}}}
  \providecommand{\lowercaseforams}{}
  \providecommand{\etc}{\emph{etc.}}\def\cprime{$'$}
\providecommand{\bysame}{\leavevmode\hbox to3em{\hrulefill}\thinspace}
\providecommand{\MR}{\relax\ifhmode\unskip\space\fi MR }
\providecommand{\MRhref}[2]{%
  \href{http://www.ams.org/mathscinet-getitem?mr=#1}{#2}
}
\providecommand{\href}[2]{#2}
\begin{thebibliography}{1}

\bibitem{Allouche-Shallit:1992:regular-sequences}
Jean-Paul Allouche and Jeffrey Shallit,
  \href{http://dx.doi.org/10.1016/0304-3975(92)90001-V}{\emph{The ring of
  $k$-regular sequences}}, Theoret. Comput. Sci. \textbf{98} (1992), no.~2,
  163--197. \MR{1166363}

\bibitem{Allouche-Shallit:2003:autom}
\bysame, \href{http://dx.doi.org/10.1017/CBO9780511546563}{\emph{Automatic
  sequences: Theory, applications, generalizations}}, Cambridge University
  Press, Cambridge, 2003. \MR{1997038 (2004k:11028)}

\bibitem{Berstel-Reutenauer:2011:noncommutative-rational-series}
Jean Berstel and Christophe Reutenauer, \emph{Noncommutative rational series
  with applications}, Encyclopedia of Mathematics and its Applications, vol.
  137, Cambridge University Press, Cambridge, 2011. \MR{2760561}

\bibitem{Dumas:2014:asymp}
Philippe Dumas,
  \href{http://dx.doi.org/10.1016/j.tcs.2014.06.036}{\emph{Asymptotic
  expansions for linear homogeneous divide-and-conquer recurrences: Algebraic
  and analytic approaches collated}}, Theoret. Comput. Sci. \textbf{548}
  (2014), 25--53.

\bibitem{Heuberger-Krenn:2018:asy-regular-sequences}
Clemens Heuberger and Daniel Krenn,
  \href{http://dx.doi.org/10.1007/s00453-019-00631-3}{\emph{Asymptotic analysis
  of regular sequences}}, Algorithmica \textbf{82} (2020), no.~3, 429--508.
  \MR{4058416}

\bibitem{Heuberger-Krenn-Lipnik:2021:asymp-analy-recur-sequen}
Clemens Heuberger, Daniel Krenn, and Gabriel~F. Lipnik,
  \href{http://dx.doi.org/10.1007/s00453-022-00950-y}{\emph{Asymptotic analysis
  of $q$-recursive sequences}}, Algorithmica \textbf{84} (2022), no.~9,
  2480--2532. \MR{4467813}

\bibitem{SageMath:2023:10.0}
{The SageMath Developers}, \emph{{SageMath} {M}athematics {S}oftware ({V}ersion
  10.0)}, 2023, \url{http://www.sagemath.org}.

\bibitem{Schuetzenberger:1961:recurrent-events}
Marcel-Paul Sch\"{u}tzenberger,
  \href{http://dx.doi.org/10.1214/aoms/1177704860}{\emph{On a special class of
  recurrent events}}, Ann. Math. Statist. \textbf{32} (1961), 1201--1213.
  \MR{133894}

\bibitem{Schuetzenberger:1961:family-of-automata}
\bysame, \href{http://dx.doi.org/10.1016/S0019-9958(61)80020-X}{\emph{On the
  definition of a family of automata}}, Information and Control \textbf{4}
  (1961), 245--270. \MR{135680}

\end{thebibliography}
\bibliographystyle{bibstyle/amsplainurl}

\end{document}